\documentclass[12pt, reqno]{amsart}
\usepackage{amsmath, amsthm, amscd, amsfonts, amssymb, graphicx, color,tikz}
\usepackage[bookmarksnumbered, colorlinks, plainpages]{hyperref}
\usepackage{mathrsfs}
\usepackage{enumerate}
\usepackage[utf8]{inputenc}
\usepackage{bbold}

\usepackage{tikz}

\usepackage[all]{xy}  

\newtheorem{theorem}{Theorem}[section]
\newtheorem{lemma}[theorem]{Lemma}
\newtheorem{proposition}[theorem]{Proposition}

\theoremstyle{definition}

\newtheorem{remark}[theorem]{Remark}

\numberwithin{equation}{section}

\newcommand{\veps}{\varepsilon}
\newcommand{\CC}{\mathbb C}

\newcommand{\DD}{\mathbb D}

\newcommand{\ZZ}{\mathbb Z}
\newcommand{\NN}{\mathbb N}
\newcommand{\NNZ}{\mathbb Z_+}
\newcommand{\NNZdhat}{\mathbb Z_+^{\hat d}}

\newcommand{\alg}{\mathrm{alg\,}}

\newcommand{\fcd}{\mathcal{F}({\mathbb{C}}^d)}
\newcommand{\vect}{\mathrm{span\,}}

\newcommand{\bd}{\mathbb B_d}
\newcommand{\hdbd}{H^2(\bd)}

\allowdisplaybreaks[4]

\hypersetup{colorlinks=true,linkcolor=blue,citecolor=green}

\begin{document}
\setcounter{page}{1}

\title[Commutants of composition operators]
{Commutants of composition operators on function spaces of several complex variables}
\date{\today}

\author[F. Bayart, M. Wang, X. Yao]{Fr\'ed\'eric Bayart, Maofa Wang and Xingxing Yao}

\address{Laboratoire de Math\'ematiques Blaise Pascal UMR 6620 CNRS, Universit\'e Clermont Auvergne, Campus universitaire des C\'ezeaux, 3 place Vasarely, 63178 Aubi\`ere Cedex, France.}
\email{frederic.bayart@uca.fr}

\address{School of Mathematics and Statistics, Wuhan University, Wuhan 430072, China.}
\email{mfwang.math@whu.edu.cn}

\address{School of Mathematics and Physics, Wuhan Institute of Technology, Wuhan 430205, China.}
\email{xxyao.math@wit.edu.cn}




\begin{abstract}
This paper is devoted to an in-depth study of the minimal commutant property for composition operators acting on Hilbert spaces of holomorphic functions in several complex variables, such as the Fock space on  $\CC^d$, the Hardy space on the Euclidean ball, and the Hardy space on the unit polydisc.
\end{abstract}
\maketitle

\tableofcontents

\section{Introduction}

Let $\mathcal{B}(H)$ denote the algebra of all bounded linear operators on a Hilbert space $H$.
Given any $A\in \mathcal{B}(H)$, recall that the commutant of $A$ is defined as the set of all operators that commute with $A$,
that is,
\begin{equation*}
  \{A\}':=\{T\in\mathcal{B}(H):TA=AT\}.
\end{equation*}
A way to investigate the commutant of $A$ is to determine if it is minimal or not. Recall that
$\{A\}'$ is a closed subalgebra of $\mathcal{B}(H)$ in the weak operator topology $\sigma$.
Let $\alg (A)$ denote the unital algebra generated by the operator $A$, that is,
\begin{equation*}
  \alg (A):=\{p(A):p\ \mbox{is a polynomial}\}.
\end{equation*}
It is easy to see that $\overline{\alg (A)}^{\sigma}$ is a commutative algebra such that $\overline{\alg (A)}^{\sigma}\subset \{A\}'$.
Then we say that an operator $A$ has a \emph{minimal commutant} if
\begin{equation*}
  \overline{\alg (A)}^{\sigma}=\{A\}'.
\end{equation*}

This property has already been investigated in several different contexts: for instance if $H$ is finite-dimensional (this amounts to saying that $A$ is nonderogatory, i.e. its eigenvalues have geometric multiplicity equal to $1$), if $A$ is a forward weighted shift on $\ell^2$ or the discrete Cesàro operator (see \cite{SW}), if $A$ is the Volterra operator (see \cite{Sa67}) or if $A$ is an analytic Toeplitz operator (see \cite{GL25}).

In this paper, we are interested in the minimal commutant property for composition operators. Recall that if $H$ is a Hilbert space of holomorphic functions defined on $\mathcal U,$ a domain of $\CC^d,$ and if $\varphi$ is a holomorphic self-map of $\mathcal U,$ the composition operator $C_\varphi$ is defined by $C_\varphi(f)=f\circ\varphi.$
We refer to the monographs written by Cowen-MacCluer \cite{cm} and
Shapiro \cite{sh} for general study of composition operators on classical spaces of analytic functions, or
to the more recent \cite{LQ26}.

After initial attempts in \cite{cl98} and in \cite{Wo02}, the minimal commutant property for composition operators on $H^2(\DD)$ was studied in depth by M. Lacruz, F. Le\'{o}n-Saavedra, S. Petrovic and L. Rodr\'{\i}guez-Piazza in \cite{LLSR18}.
They proved a complete characterization of the symbols $\varphi$ such that $C_\varphi$ has the minimal commutant property provided $\varphi$ is a linear fractional map.
More precisely, if $\varphi$ is a linear fractional self-map of the unit disk $\mathbb{D}$, then on $H^{2}(\mathbb{D})$, $C_{\varphi}$ has a minimal commutant
if and only if $\varphi$ is a non-periodic elliptic automorphism,
or a parabolic non-automorphism, or a loxodromic map.

The methods of \cite{LLSR18} were late extended in \cite{BY} in the context of Hardy spaces of Dirichlet series,
with applications to the classical Hardy space $H^2(\DD)$.

\smallskip

The aim of this paper is to investigate the minimal commutant property for composition operators acting on spaces of holomorphic functions of several complex variables. The key point in  \cite{LLSR18} and \cite{BY} was the existence of a Koenigs function, i.e. a holomorphic map $\sigma$ such that $\sigma\circ\varphi=\lambda\sigma.$ Working in several complex variables, we must replace the complex number $\lambda$ by the matrix $A=\varphi'(0)$ and the existence of such a
$\sigma$ is not always guaranteed. This leads us to introduce a general framework: we say that $(H,\varphi)$ is a strong Koenigs system if there exists a Koenigs map $\sigma$ (i.e. satisfying $\sigma\circ\varphi=A\circ\sigma$) and if some natural assumptions
are satisfied. For a strong Koenigs system, we characterize whether $C_\varphi$ admits the minimal commutant property (see Theorem \ref{thm:sufficiency}). This characterization involves the structure of the eigenvalues of $A,$ in particular the size of its Jordan blocks, a feature absent from the one-variable theory.
In particular, we characterize completely the minimal commutant property for compact composition operators acting on the Fock space in Section \ref{sec:fock} (see Theorem \ref{thm:fock}), and provide interesting examples for the Hardy spaces of the Euclidean unit ball or of the unit polydisc in Section \ref{sec:ball-polydisc} (see Theorems \ref{thm:ball} and \ref{thm:polydisc}).

\smallskip

{\bf Notation.} Throughout this work, $C$ and $M$ will denote positive constants (which could depend on the ambient Hilbert space of holomorphic functions, but not on the functions or the multi-indices that will be involved). The values of $C$ and $M$ may change from line to line. For a map $f:E\to\CC^d$ and $\alpha\in\ZZ_+^d,$ $f^\alpha:E\to\mathbb C$ is defined by $f^\alpha=f_1^{\alpha_1}\cdots f_d^{\alpha_d}.$

\section{Preliminaries}\label{Pre}

\subsection{Some auxiliary lemmas}

In this subsection, we gather some results on operators with a minimal commutant.

\begin{lemma}$^{\mbox{\tiny \cite[Lemma 1.3]{LLSR18}}}$\label{sum-com}
Consider a direct sum $H=H_{1}\oplus H_{2}$, let $A_{j}\in \mathcal{B}(H_{j})$ for $j=1,2$, and let $A=A_{1}\oplus A_{2}$.
If $A$ has a minimal commutant, then so do both $A_{1}$ and $A_{2}$.
\end{lemma}

\begin{lemma}$^{\mbox{\tiny \cite[Lemma 1.4]{LLSR18}}}$\label{sim-com}
If $A\in \mathcal{B}(H)$ has a minimal commutant, $S$ is an invertible operator and $B=S^{-1}AS$, then $B$ also has a minimal commutant.
\end{lemma}

If the space $H$ is finite-dimensional, the operators $T\in\mathcal B(H)$ with the minimal commutant property are easily characterized. Observe that in that case, the minimal commutant property reduces to the equality $\textrm{alg}(A)=\{A\}'$. Recall that an operator is nonderogatory if all its eigenvalues have geometric multiplicity $1.$

\begin{lemma}$^{\mbox{\tiny \cite[Corollary 4.4.18]{HJ91}}}$\label{lem:finitedimension} Assume that $\dim(H)<\infty$ and let $T\in\mathcal B(H).$ Then $T$ has a minimal commutant if and only if $T$ is nonderogatory.
\end{lemma}

\begin{lemma}\label{lem:powercommutant}
Let $T\in\mathcal{B}(H)$ and $k\in\NN$. If $T^k$ has a minimal commutant, then $T$ has a minimal commutant.
\end{lemma}
\begin{proof}
Indeed, it is clear that $\{T\}'\subset \{T^k\}'$ and conversely that $\textrm{alg}(T^k)\subset \textrm{alg}(T).$
\end{proof}

\subsection{Polynomial interpolation}

As an application of Mergelyan's theorem, we prove the following interpolation result,
which will play a critical role in constructing our operator approximation.

\begin{lemma}\label{inter-lem}
Let $(z_{j})_{j\geq 1}\subset\DD$ be  such that $(|z_{j}|)$ is decreasing, let $(n_{j})_{j\geq 1}\subset\mathbb{N}$ and let $(a_{j,k})_{j\geq1,\ 0\leq k\leq n_{j}-1}\subset\mathbb{C}$.
Let also $\varepsilon>0$ and $M\in\NN$. For any $r\in(0,1)$ and any $N\in\NN$ such that $|z_j|>r$ if $1\leq j\leq N$ and $|z_j|\leq r$ if $j\geq N+1,$ there exists a holomorphic polynomial $P$ such that:
$$\left\{\begin{array}{rcll}
|P^{(k)}(z_j)-a_{j,k}|&\leq& \veps  & \textrm{for all } 1\leq j\leq N, 0\leq k \leq n_j-1. \\ [0.3cm]
|P(z)| &\leq& \varepsilon |z|^M & \textrm{for } |z|\leq r.
\end{array}\right.$$
\end{lemma}

\begin{proof}
Let $(r_j)_{j=1,\dots,N}$ be such that the discs $\overline{D}(z_j,r_j)$ are pairwise disjoint and contained in $\overline{\DD}\backslash r\overline{\DD}.$ Note that $K=r\overline{\DD}\cup\bigcup_{j=1}^N \overline{D}(z_j,r_j)$ is compact and has a connected complement. Let $f$ be defined on $K$ by
\begin{equation*}
f(z)=
\left\{
\begin{array}{cl}
\displaystyle \frac1{z^M}\sum_{k=0}^{n_j-1}\frac{a_{j,k}}{k!}(z-z_j)^{k}&\textrm{ if }z\in\overline{D}(z_j,r_j)\\[0.3cm]
0&\textrm{ if }z\in r\overline{\DD}.
\end{array}
\right.
\end{equation*}
The function $f$ is continuous on $K$ and holomorphic inside $K.$ By Mergelyan's theorem, we can choose a holomorphic polynomial $Q$ which is arbitrarily close to $f$ on $K$. Then $P=z^M Q$ satisfies the required conditions, thanks to Cauchy's estimates.
\end{proof}

\section{Minimal commutant of composition operators: Necessity}\label{com-nec}

This section provides some necessary conditions for a composition operator to have a minimal commutant.
We say that the complex numbers $\lambda_1,\dots,\lambda_s$ are \emph{independent} if, for any $\alpha,\beta\in\ZZ_+^s,$ $\lambda^\alpha=\lambda^\beta$ implies $\alpha=\beta.$ In particular, if $\lambda_i=\lambda_j$ for some $i\neq j$ or if $\lambda_i=0$ for some $i,$ then $\lambda_1,\dots,\lambda_s$ are not independent.

Let $\mathcal{U}$ be a domain in $\mathbb{C}^d$ with $0 \in \mathcal{U}$. Let $H$ be a reproducing kernel Hilbert space of holomorphic functions on $\mathcal{U}$ such that $z^\alpha \in H$ for all $\alpha \in \mathbb{Z}_+^d$. Let $\varphi: \mathcal{U} \to \mathcal{U}$ be a holomorphic map, and set $A = \varphi'(0)$.

We say that $(H, \varphi)$ is a \emph{Koenigs system} if the following conditions hold:

\begin{itemize}
\item $\varphi$ induces a bounded composition operator on $H$;
    \item There exists a univalent map $\sigma: \mathcal{U} \to \mathbb{C}^d$ such that $\sigma \circ \varphi = A \circ \sigma$ and $0\in\sigma(\mathcal U)$;
    \item $\sigma^\alpha \in H$ for all $\alpha \in \mathbb{Z}_+^d$.
\end{itemize}
Note that the second condition implies that, for any $f\in H,$ the map $f\circ\sigma^{-1}$ is holomorphic at $0.$  Hence there exists a unique sequence of homogeneous polynomials $(P_m(f))_{m \geq 0}$ with $\deg(P_m(f)) = m$ and a neighbourhood $\mathcal{V}$ of $\sigma^{-1}(0)$ such that, for all $z \in \mathcal{V}$,
\[
f(z) = \sum_{m=0}^{\infty} P_m(f) \circ \sigma(z).
\]
In other words, $P_m(f)$ is the homogeneous polynomial of degree $m$ appearing in the Taylor expansion of $f\circ\sigma^{-1}$  at $0$.
Moreover, the intertwining equation $\sigma \circ \varphi = A \circ \sigma$ and the uniqueness of the sequence $(P_m(f))$ imply that  $P_m(f \circ \varphi) = P_m(f) \circ A$ for any $m \in \ZZ_+$ and any $f \in H$.

\begin{theorem}\label{nece}
Let $(H,\varphi)$ be a Koenigs system and set $A=\varphi'(0)$. If $C_{\varphi}$ has a minimal commutant, then the following conditions hold:
\begin{itemize}
  \item[$\bullet$] The canonical Jordan form of $A$ admits at most one Jordan block whose size exceeds $1$, and the size of this Jordan block, if it exists,  is exactly $2$;
  \item[$\bullet$] If $\lambda_1,\dots,\lambda_{\hat d}$ are the eigenvalues of $A$, listed with repetition according to their geometric multiplicities, then they are independent.
\end{itemize}
\end{theorem}

We may observe that the first condition implies that $\hat d=d$ or $d-1$ and that $\hat d$ is the number  of Jordan blocks of $A$.
\begin{proof}
For $N \geq 0$, let
\begin{align*}
\mathcal{P}_N& = \left\{ P \circ \sigma : P \in \mathbb C [z_1, \dots, z_d],\ \text{deg}(P) \leq N \right\}\\
\mathcal{Q}_N& = \bigcap_{m=0}^{N} \ker(P_m),
\end{align*}
where $P_m$ is the linear map from $H$ to $\mathbb C[z_1,\dots,z_d]$ sending $f$ to $P_m(f).$
Note that
\begin{itemize}
\item $\mathcal{P}_N$ is $C_\varphi$-invariant, since $\sigma \circ \varphi = A \circ \sigma$.
\item $\mathcal{Q}_N$ is $C_\varphi$-invariant, since $P_m(f \circ \mathcal{\varphi}) = P_m(f) \circ A$ for $m=0,\dots,N$.
\item $\mathcal{P}_N\cap \mathcal{Q}_N=\{0\}$: this follows from $f=\sum_{m=0}^N P_m(f)\circ\sigma$ for all $f\in \mathcal{P}_N$.
\item $H = \mathcal{P}_N \oplus \mathcal{Q}_N$, writing any $f\in H$ as
$$f=\sum_{m=0}^N P_m(f)\circ\sigma+\left(f-\sum_{m=0}^N P_m(f)\circ\sigma\right).$$
\end{itemize}
 Therefore, by Lemma \ref{sum-com}, the restriction ${C_{\varphi}}_{|\mathcal{P}_N}$ also has a minimal commutant. Since $\mathcal P_N$ is finite-dimensional, ${C_{\varphi}}_{|\mathcal P_N}$ must be nonderogatory.

Assume first that the canonical Jordan form of $A$ admits two Jordan blocks of size at least $2$ associated to the eigenvalues $\lambda$ and $\mu.$
Let $x_1^*, x_2^*, x_3^*, x_4^* \in (\mathbb{C}^d)^*$ be such that
$$\begin{array}{lll}
x_1^* \circ A = \lambda x_1^*&&x_3^* \circ A = \mu x_3^*\\
x_2^* \circ A = \lambda x_2^*+x_1^*&&
x_4^* \circ A = \mu x_4^* + x_3^*.
\end{array}$$

We transfer these eigenvectors to $C_\varphi$ by setting $L_i = x_i^* \circ \sigma$, $i=1,\dots,4$, to get
$$\begin{array}{lll}
C_\varphi(L_1) = \lambda L_1&&C_\varphi(L_3) = \mu L_3\\
C_\varphi(L_2) = \lambda L_2 + L_1&&
C_\varphi(L_4) = \mu L_4 + L_3.
\end{array}$$

Then, let us set $P_1=L_1L_3$ and $P_2=L_2L_3-\frac{\mu}{\lambda}L_1L_4$ if $\lambda\neq 0,$ and $P_1=L_1^2,$ $P_2=L_1L_2$ if $\lambda=0$. It is easy to show that $P_1,P_2\in\mathcal P_2$ and that $(P_1,P_2)$ is a linearly independent family of eigenvectors of $C_\varphi$ associated to $\lambda\mu,$
which contradicts the nonderogatoriness of ${C_{\varphi}}_{|\mathcal P_2}$.

If the canonical Jordan form of $A$ admits one Jordan block of size at least 3, then there exist $L_1, L_2, L_3 \in \mathcal{P}_1$ such that
\[
C_\varphi(L_1) = \lambda L_1, \quad C_\varphi(L_2) = \lambda L_2 + L_1, \quad C_\varphi(L_3) = \lambda L_3 + L_2.
\]
Then, set $P_1=L_1^2$ and $P_2=L_1L_2+2\lambda L_1L_3-\lambda L_2^2$. It follows that $(P_1,P_2)$ is a linearly independent family of eigenvectors of $C_\varphi$ associated with $\lambda^2,$ which is again a contradiction.

Finally, assume that there exist $\alpha \neq \beta$ such that $\lambda^{\alpha} = \lambda^{\beta}$.
Let $x_1^*, \ldots, x_{\hat d}^*$ be such that $x_i^* \circ A = \lambda_i x_i^*$ for all $i$ and consider $L_1 = (x^*)^\alpha \circ \sigma$, $L_2 = (x^*)^\beta \circ \sigma$. Then $L_1, L_2 \in \mathcal{P}_N$ with $N = \max( |\alpha|, |\beta|)$ and
$(L_1, L_2)$ is a family of linearly independent vectors associated to the same eigenvalue of ${C_\varphi}_{|\mathcal{P}_N}$.
 This is again a contradiction.
\end{proof}

\section{Minimal commutant of composition operators: Sufficiency}\label{com-suf}

\subsection{Strong Koenigs systems}

To obtain a sufficient condition for the minimal commutant property, we require additional assumptions on the Koenigs system $(H,\varphi).$ We start with a Koenigs system $(H,\varphi)$ with $A=\varphi'(0)$ and Koenigs map $\sigma$. For $n\in\ZZ_+,$ we define $\mathcal F_n=\vect(\sigma^\alpha:\ |\alpha|=n)$ and $\mathcal G_n=\vect(z^\alpha:\ |\alpha|=n)$. We say that $(H,\varphi)$ is a \emph{strong Koenigs system} if the following conditions hold:
\begin{enumerate}
\item[(SK1)] For any compact set $K\subset\mathcal U,$ $\sup_{z\in K}\|k_z\|<\infty,$ where $k_z$ is the reproducing kernel at $z\in\mathcal U.$
\item[(SK2)] For any basis $(x_1^*,\dots,x_d^*)$ of $(\mathbb C^d)^*,$ denoting by $L_i=x_i^*\circ\sigma,$ there exist $C,M>0$ such that, for all $f\in H,$ there exist a neighbourhood $\mathcal V$ of $\sigma^{-1}(0)$ and a sequence of complex numbers $(a_\alpha)_{\alpha\in\NNZ^d}$ such that, for all $z\in\mathcal V,$
$$f(z)=\sum_{\alpha\in\NNZ^d} a_\alpha L^ \alpha(z)$$
with absolute and uniform convergence on compact subsets‌ of $\mathcal V$, and for all $\alpha\in\NNZ^d,$
\begin{equation}\label{eq:sks4}
\|a_\alpha L^\alpha\|\leq C M^{|\alpha|}\|f\|.
\end{equation}
\item[(SK3)] There exist $C,M>0$ such that, for all $n\geq 0,$
\begin{eqnarray}
\left\|C_{\sigma:\mathcal G_n\to\mathcal F_n}\right\|&\leq& C M^n \label{eq:sks5} \\
\left\|C_{\sigma^{-1}:\mathcal F_n\to\mathcal G_n}\right\|&\leq& C M^n. \label{eq:sks6}
\end{eqnarray}
\item[(SK4)] There exist $C>0$ and $a\in(0,1)$ such that, for all $n\geq 0,$
\begin{equation}
\left\|C_{A:\mathcal G_n\to\mathcal G_n}\right\|\leq C a^n. \label{eq:sks7}
\end{equation}
\item[(SK5)] There exists a domain $\Omega\subset\CC$ such that, for all $w\in\Omega,$ $w\sigma(\mathcal U)\subset \sigma(\mathcal U)$, $[0,1)\subset\Omega$ and, setting $\varphi_w=\sigma^{-1}(w\sigma)$ for all $w\in\Omega,$
\begin{equation}
\sup_{w\in[0,1)}\|C_{\varphi_w}\|<\infty  \label{eq:sks2} \\
\end{equation}
\begin{equation}
\forall f,g\in H,\ w\in\Omega\mapsto  \langle C_{\varphi_w}(f),g\rangle\textrm{ is holomorphic.}  \label{eq:sks3}
\end{equation}
\end{enumerate}

\bigskip

These conditions may not be particularly appealing. However, some of them are consequences of simpler assumptions.

\begin{lemma}\label{lem:suf1}
Let $(H,\varphi)$ be a Koenigs system satisfying {\normalfont (SK1)}.  Assume that there exist $C,M>0$ such that, for all $\alpha\in\ZZ_+^d,$
\begin{equation}\label{eq:sks1}
\|z^\alpha\|_H\leq CM^{|\alpha|}\textrm{ and }\|\sigma^\alpha\|_H\leq CM^{|\alpha|}.
\end{equation}
Then {\normalfont (SK2)} and {\normalfont (SK3)} hold.
\end{lemma}

\begin{proof}
Let $f\in H.$ There exists a neighbourhood $\mathcal V_1$ of $0$ such that, for all $z\in\mathcal V_1,$
$$f\circ\sigma^{-1}(z)=\sum_{\alpha\in\NNZ^d}b_\alpha z^\alpha,$$
where $b_\alpha=\frac{\partial^\alpha (f\circ\sigma^{-1})(0)}{\alpha!}.$ Hence, there exists a neighbourhood $\mathcal V_2$ of $\sigma^{-1}(0)$ such that for all $z\in\mathcal V_2,$
$$f(z)=\sum_{\alpha\in\NNZ^d}b_\alpha \sigma^{\alpha}(z).$$
Let $r>0$ be such that $\sigma^{-1}(r\overline{\DD}^d)\subset\mathcal U.$ By Cauchy's inequalities, for all $\alpha\in\NNZ^d,$
\begin{align}
|b_\alpha|&\leq \frac1{r^{|\alpha|}} \sup_{z\in\sigma^{-1}(r\overline{\DD}^d)} |f(z)| \nonumber \\
&\leq C M^{|\alpha|} \|f\|. \label{eq:estimatecoef}
\end{align}
Let $e^*=(e_1^*,\dots,e_d^*)^T$ where $(e_1^*,\dots,e_d^*)$ is the canonical basis of $(\mathbb C^d)^*$, and let $x^*=(x_1^*,\dots,x_d^*)^T.$ Let $P\in GL_d(\CC)$ be such that $e^*=Px^*$ and let $P_1,\dots,P_d$ be the rows of $P.$ Then
$$f=\sum_{\alpha\in\NNZ^d}b_\alpha\prod_{j=1}^d (P_jx^*\circ\sigma)^{\alpha_j}.$$
For a fixed multi-index $\alpha\in\NNZ^d$ with $|\alpha|=n,$  expanding the product, we find a sum of $d^n$ terms, each of which can be written as $c_\beta (x^*\circ\sigma)^{\beta}$ with $|\beta|=n$ and
$|c_\beta|\leq  p_{\textrm{max}}^n |a_\alpha|,$ where $p_{\max}=\max_{i,j}|p_{i,j}|.$ Since there are $\binom{n+d-1}d\leq n^d\leq M^n$ such multi-indices, grouping the terms together gives
$$f(z)=\sum_{\alpha\in\NNZ^d}a_\alpha L^\alpha(z)$$
with $|a_\alpha|\leq C M^{|\alpha|}\|f\|$ for some $C,M>0,$ with absolute and uniform convergence on compact subsets‌ of the intersection
of $\mathcal V_2$ and of some polydisc centered at $\sigma^{-1}(0).$ Moreover, we write
$$L^\alpha=\prod_{j=1}^d (Q_j e^*\circ\sigma)^{\alpha_j},$$
where $Q=P^{-1}$ and $Q_1,\dots,Q_d$ are the rows of $Q.$ Arguing as above, we get
$$\|L^\alpha\|\leq CM^{|\alpha|}$$
from the second half of \eqref{eq:sks1}, which yields \eqref{eq:sks4}.

Let now $n\geq 0$ and $f=\sum_{|\alpha|=n}b_\alpha \sigma^\alpha\in\mathcal F_n.$ Then, by \eqref{eq:estimatecoef}, $|b_\alpha|\leq C M^n \|f\|$ for all $\alpha$ with $|\alpha|=n$ and we conclude that
$$\|C_{\sigma^{-1}}(f)\|\leq \sum_{|\alpha|=n} |b_\alpha|  \|z^\alpha\|\leq CM^n\|f\|$$
by the first half of  \eqref{eq:sks1}.
Finally, let $g=\sum_{|\alpha|=n}a_\alpha z^\alpha\in\mathcal G_n.$ Again, Cauchy's inequalities yield
the existence of $C,M>0$ such that  $|a_\alpha|\leq CM^n \|g\|$ for all $\alpha\in\NNZ^d$ with $|\alpha|=n.$
Then
$$\|C_\sigma(g)\|\leq \sum_{|\alpha|=n} |a_\alpha| \|\sigma^\alpha\|\leq C M^n \|g\|$$
by the second half of \eqref{eq:sks1}.
\end{proof}

\begin{lemma}\label{lem:suf2}
{Let $(H,\varphi)$ be a Koenigs system with $A=\varphi'(0).$}
Assume that there exists $\lambda>1$ such that $C_{\lambda A}$ is bounded  on $H.$ Then
{\normalfont (SK4)} is satisfied.
\end{lemma}
\begin{proof}
Let $n\geq 0$ and $g=\sum_{|\alpha|=n}a_\alpha z^\alpha\in\mathcal G_n.$
Then
$$C_A(g)=\frac1{\lambda^n} C_{\lambda A}(g)$$
and we get
\begin{equation*}
\left\|C_{A:\mathcal G_n\to\mathcal G_n}\right\|\leq \|C_{\lambda A}\|\left(\frac 1{\lambda}\right)^n.
\end{equation*}
\end{proof}

\subsection{Statement and proof of the main result}

\begin{theorem}\label{thm:sufficiency}
Let $(H,\varphi)$ be a strong Koenigs system and set $A=\varphi'(0).$ Assume that the spectral radius of $A$ is less than $1$ and that
\begin{itemize}
  \item[$\bullet$] the canonical Jordan form of $A$ admits at most one Jordan block of size greater than $1$, and the size of this Jordan block, if it exists, is exactly $2$;
  \item[$\bullet$] if $\lambda_1,\dots,\lambda_{\hat d}$ are the eigenvalues of $A$, listed according to their geometric multiplicities, then they are independent.
\end{itemize}
Then $C_\varphi$ has a minimal commutant.
\end{theorem}

The remainder of this section is devoted to proving this theorem. We start with a strong Koenigs system satisfying the assumptions of the previous theorem and we consider $x_1^*,\dots,x_d^*$ a basis of $(\CC^d)^*$ such that
\begin{itemize}
\item[$\bullet$] $x_i^*\circ A=\lambda_i x_i^*$ for all $i=1,\dots,\hat d$.
\item[$\bullet$] $x_d^*\circ A=\lambda_{d-1}(x_d^*+x_{d-1}^*)$ if $\hat d=d-1.$
\end{itemize}
For $i=1,\dots,d$ we set $L_i=x_i^*\circ\sigma$ and for $\beta\in\NNZdhat,$ define
\begin{itemize}
\item[$\bullet$] $m(\beta)=0$ if $\hat d=d;$  $m(\beta)=\beta_{d-1}$ otherwise.
\item[$\bullet$] $\Gamma(\beta)=\{\beta\}$ if $\hat d=d;$
$$\Gamma(\beta)=\{\alpha\in\NNZ^d:\ \alpha_1=\beta_1,\dots,\alpha_{d-2}=\beta_{d-2},\alpha_{d-1}+\alpha_d=\beta_{d-1}\}$$
otherwise.
\item[$\bullet$] $\mathcal P_\beta=\vect(L^\alpha:\ \alpha\in \Gamma(\beta))$ and $Q_\beta=\vect((x^*)^\alpha:\ \alpha\in\Gamma(\beta)).$
\end{itemize}
We  observe that $\mathcal P_\beta$ is $C_{\varphi_w}$-invariant and  $C_{\varphi_w}(f)=w^{|\beta|}f$ for any $f\in \mathcal P_\beta$ and any $w\in\Omega.$

\medskip

The proof of Theorem \ref{thm:sufficiency} will require two propositions. The first one deals about the generalized eigenspaces of $C_\varphi.$
\begin{proposition}\label{prop:eigenspace}
Under the assumptions of Theorem \ref{thm:sufficiency}, for all $\beta\in\NNZdhat,$
$\ker((C_\varphi-\lambda^\beta I)^{m(\beta)+1})=\mathcal P_\beta.$
\end{proposition}
\begin{proof}
We assume $\hat d=d-1.$ The case $\hat d=d$ is similar (and even simpler) and left to the reader. First note that $\mathcal P_\beta$ is $C_\varphi$-invariant.  Moreover,
$L_1^{\beta_1}\cdots L_{d-2}^{\beta_{d-2}}L_{d-1}^{m(\beta)-k}L_d^k,$ $k=0,\dots,m(\beta),$ is a basis of $\mathcal P_\beta.$ The matrix of $C_{\varphi|\mathcal P_\beta}$ with respect to this basis is
\begin{equation}
\lambda^\beta\begin{pmatrix}
1&&\binom ji\\
&\ddots&\\
0&&1
\end{pmatrix},\ 0\leq i\leq j\leq m(\beta).
\label{eq:propeigenspace}
\end{equation}
This shows that $(C_\varphi-\lambda^\beta I)^{m(\beta)+1}=0$ on $\mathcal P_\beta,$
that is, $\mathcal P_\beta\subset\ker((C_\varphi-\lambda^\beta I)^{m(\beta)+1})$.

For the reverse inclusion, since $\dim(\mathcal P_\beta)=m(\beta)+1,$ by the theory of iterated kernels, it suffices to show that $\dim(\ker(C_\varphi-\lambda^\beta I))\leq 1.$ Let $f\in \ker(C_\varphi-\lambda^\beta I)$, and write $A=P^{-1}JP$ with $P\in GL_d(\CC)$ and
$J$ in Jordan form,
$$J=\begin{pmatrix}
\lambda_1&&&0\\
&\ddots&&\\
&&\lambda_{d-1}&\lambda_{d-1}\\
0&&&\lambda_{d-1}
\end{pmatrix}.$$
Set $g=f\circ (P\circ \sigma)^{-1}$ which is holomorphic at   $0\in\sigma(\mathcal U).$
Then, in a neighbourhood of $0,$ using $\sigma^{-1}\circ A=\varphi\circ\sigma^{-1},$
\begin{align*}
g\circ J&=f\circ(P\circ \sigma)^{-1}\circ J\\
&=f\circ\sigma^{-1}\circ A\circ P^{-1}\\
&=f\circ\varphi\circ(P\sigma)^{-1}\\
&=\lambda^\beta g.
\end{align*}
 For $\gamma\in\NNZdhat,$  set $\mathcal G_\gamma=\vect(z^\alpha:\ \alpha\in\Gamma(\gamma))$. Then $\mathcal G_\gamma\subset \mathcal G_{|\gamma|}$, $C_J(\mathcal G_\gamma)\subset\mathcal G_\gamma$ and  the matrix of $C_{J|\mathcal G_\gamma}$ with respect to the basis $z^\alpha,$ $\alpha\in\Gamma(\gamma)$, is of the form given in \eqref{eq:propeigenspace}, with $\beta$ replaced by $\gamma.$ In particular, $C_{J|\mathcal G_\gamma}$ has a unique eigenvalue, $\lambda^\gamma,$ and the associated eigenspace is spanned by $z_1^{\gamma_1}\cdots z_{d-1}^{\gamma_{d-1}}.$ Moreover, any function $h$ holomorphic in a neighbourhood of $0$ can be uniquely decomposed as
$$h(z)=\sum_{n\in\NNZ}\sum_{\substack{|\gamma|=n\\ \gamma\in\NNZdhat}}h_\gamma(z)$$
with absolute and uniform convergence on compact subsets‌ of a neighbourhood of $0,$ where $h_\gamma\in\mathcal G_\gamma$.
Since $g\circ J=\lambda^\beta g,$ we get, for all $\gamma\in\NNZdhat,$
$$(g\circ J)_\gamma=\lambda^\beta g_\gamma.$$
On the other hand,
$$g\circ J=\sum_{n\in\NNZ}\sum_{\substack{|\gamma|=n\\ \gamma\in\NNZdhat}}g_\gamma\circ J$$
so that for all $\gamma\in\NNZdhat,$ $g_\gamma\circ J=\lambda^\beta g_\gamma.$
This implies that $g_\gamma=0$ if $\gamma\neq\beta$ and $g=g_\beta\in\vect(z_1^{\beta_1}\cdots z_{d-1}^{\beta_{d-1}}).$ Since the linear map $f\mapsto g$ is injective, we conclude that $\dim(\ker(C_\varphi-\lambda^\beta I))\leq 1.$
\end{proof}

We now consider $T\in\{C_\varphi\}'$ and observe that the previous proposition implies that
$T(\mathcal P_\beta)\subset \mathcal P_\beta$ for all $\beta\in\NNZdhat.$
Since the matrix of $C_{\varphi|\mathcal P_\beta}$ in the basis described above is given by
\eqref{eq:propeigenspace}, one may find a basis $\mathcal B_\beta$ of $\mathcal P_\beta$ such that
the matrix of $C_{\varphi|\mathcal P_\beta}$ with respect to $\mathcal B_\beta$ is the Jordan block
$$J(\lambda^\beta,m(\beta)+1)=
\begin{pmatrix}
\lambda^{\beta}&1&\dots&0\\
0&\ddots&\ddots&\vdots\\
\vdots&\ddots&\ddots&1\\
0&\dots&\dots&\lambda^{\beta}
\end{pmatrix}.$$
The commutant of a Jordan block is well-known (for example see \cite{GLR}),
and this implies that there exist $c_{\beta,0}, \dots,c_{\beta,m(\beta)}\in\CC$ such that {the matrix of $T_{|\mathcal P_\beta}$ with respect to the basis $\mathcal B_\beta$ is}
$$\textrm{Mat}({T}_{|\mathcal{P}_{\beta}},\mathcal B_{\beta})=\begin{pmatrix}
c_{\beta,0}&c_{\beta,1}&\dots&c_{\beta,m(\beta)}\\
0&\ddots&\ddots&\vdots\\
\vdots&\ddots&\ddots&c_{\beta,1}\\
0&\dots&0&c_{\beta,0}
\end{pmatrix}.$$
This is one of the main tool to prove our next proposition.

\begin{proposition}\label{prop:approxim}
Under the assumptions of Theorem \ref{thm:sufficiency},  there exists $\delta>0$ such that, for all $T\in\{C_\varphi\}',$ for all $\veps>0,$ for all $w\in\delta\overline{\DD},$
there exists a polynomial $P\in\CC[X]$ satisfying
$$\|P(C_\varphi)-TC_{\varphi_w}\|<\veps.$$
\end{proposition}

\begin{proof}
First choose $C_1,M_1>0$ and $a\in(0,1)$ such that \eqref{eq:sks4}, \eqref{eq:sks5}, \eqref{eq:sks6} and \eqref{eq:sks7} hold.
 Let $f\in H.$ By (SK2), we can write for $z\in\mathcal V,$
$$f(z)=\sum_{\alpha\in\NNZ^d}a_\alpha L^\alpha(z).$$
For $\beta\in\NNZdhat,$ set $f_\beta=\sum_{\alpha\in\Gamma(\beta)}a_\alpha L^\alpha$. Since $\textrm{card}(\Gamma(\beta))\leq |\beta|,$ we may assume that
$\|f_\beta\|\leq C_1 M_1^{|\beta|}\|f\|,$ enlarging $C_1$ and $M_1$ if necessary.
Choose $\delta>0$ such that $\delta M_1<1$ and $N\in\NN$ such that $M_1^3 a^N<1.$
Let $r,\eta\in(0,1)$ be fixed later.  We set
\begin{align*}
I_r&=\{\beta\in\NNZdhat:\ |\lambda^\beta|\geq r\}\\
J_r&=\{\beta\in\NNZdhat:\ |\lambda^\beta|<r\}.
\end{align*}
Since the spectral radius of $A$ is less than $1$, $I_r$ is finite and there exists $\rho\in(0,r)$ such that
$$J_r=\{\beta\in\NNZdhat:\ |\lambda^\beta|\leq \rho\}.$$
Set
$$K_r=\left\{\beta\in\NNZdhat:\ C_1 a^{|\beta|}\geq \frac{r+\rho}2\right\},$$
which is also finite and contains $I_r,$ enlarging $a\in(0,1)$ if necessary.
Let $w\in\Omega$. By Lemma \ref{inter-lem}, there exists a polynomial $P\in\CC[X]$ such that
$$\left|
\frac{P^{(k)}(\lambda^\beta)}{k!}-c_{\beta,k}w^{|\beta|}\right|\leq \eta$$
for all $\beta\in I_r$ and all $k=0,\dots,m(\beta)$, and such that
$$|P(z)|\leq\eta |z|^N$$
for $|z|\leq (r+\rho)/2.$ We finally set
$$
\begin{array}{rclcrcl}
F_r&=&\sum_{\beta\in I_r}f_\beta&&G_r&=&f-F_r\\
\widetilde{F_r}&=&\sum_{\beta\in K_r}f_\beta&&\widetilde{G_r}&=&f-\widetilde{F_r}.
\end{array}$$

The proof of Proposition \ref{prop:approxim} will need the following four   facts.

\medskip

\noindent{\bf Fact 1.} There exists $r_1>0$ such that for all $r\in(0,r_1]$,$w\in\delta\overline{\DD}$ and $f\in H,$
$$\|TC_{\varphi_w}(G_r)\|\leq\veps \|f\|.$$

\begin{proof}[Proof of Fact 1]
 For $z\in\mathcal V$, we may write
 $$G_r(z)=\sum_{\beta\in J_r}f_\beta(z).$$
 Since for all $\alpha\in\NNZ^d$ and  $w\in\Omega,$
 $$C_{\varphi_w}(L^\alpha)=(x^*)^\alpha\circ\sigma\circ\sigma^{-1}(w\sigma)=w^{|\alpha|}L^\alpha,$$
 we get
 $$C_{\varphi_w}(G_r)=\sum_{\beta\in J_r}w^{|\beta|} f_\beta(z).$$
 Now,
 $$\|w^{|\beta|}f_\beta\|\leq C_1 (\delta M_1)^{|\beta|}\|f\|.$$
 Since $\delta M_1<1,$ the series $\sum_{\beta\in J_r} w^{|\beta|}f_\beta$ is convergent in $H$ and we get
 $$\|TC_{\varphi_w}(G_r)\|\leq C_1\|T\|\sum_{\beta\in J_r}(\delta M_1)^{|\beta|}\|f\|.$$
 Let $N_r=\inf\{|\beta|:\ \beta\in J_r\}.$ Then $(N_r)\to\infty$ as $r\to 0$. So
 \begin{align*}
 \|TC_{\varphi_w}(G_r)\|&\leq C_1\|T\|
 \sum_{n\geq N_r}\sum_{\begin{subarray}{c}
 |\beta|=n\\
 \beta\in\NNZdhat
 \end{subarray}}(\delta M_1)^n \|f\|\\
 &\leq C_1\|T\|\sum_{n\geq N_r}\binom{n+\hat d-1}{\hat d}(\delta M_1)^n \|f\|,
 \end{align*}
 and the last term is less than $\veps\|f\|$ for sufficiently small $r$.
\end{proof}

\medskip

\noindent{\bf Fact 2.} There exists $r_2\in(0,1)$ such that for all $r\in(0,r_2]$ and $f \in H,$
$$\left\|P(C_{\varphi})\left(\widetilde{G_r}\right)\right\|\leq\veps\|f\|.$$
\begin{proof}[Proof of Fact 2]
Let $\beta\notin K_r$ and $f\in H.$ We observe that for all $z\in\mathcal V,$
$$P(C_\varphi)(f_\beta)(z)=C_\sigma\circ P(C_A)\circ C_{\sigma^{-1}}(f_\beta)(z),$$
where $C_{\sigma^{-1}}(f_\beta)\in\mathcal G_{|\beta|}.$ Set
$$S=\frac{C_{A|\mathcal G_{|\beta|}}}{C_1 a^{|\beta|}}.$$
Then $S$ is a contraction on $\mathcal G_{|\beta|}.$ Define $R(z)=P(C_1 a^{|\beta|} z).$ By von Neumann's inequality,
\begin{align*}
\|P(C_A)_{|\mathcal G_{|\beta|}\to \mathcal G_{|\beta|}}\|&=\|R(S)\| \\
&\leq \sup_{|z|=1} |P(C_1 a^{|\beta|}z)|\\
&\leq C_1^N a^{N|\beta|}
\end{align*}
since $C_1 a^{|\beta|}<(r+\rho)/2$ for $\beta\notin K_r.$
By \eqref{eq:sks5} and \eqref{eq:sks6} and since $\mathcal G_{|\beta|}$ is $P(C_A)-$invariant, this yields
\begin{align*}
\|P(C_\varphi)(f_\beta)\|&\leq C_1^2 M_1^{2|\beta|}C_1^N a^{N|\beta|} \|f_\beta\|\\
 &\leq C_1^{N+3}(M_1^3a^N)^{|\beta|}\|f\|.
\end{align*}
The proof now concludes as in Fact 1.
\end{proof}

Fix  $r=\min(r_1,r_2)$ and we show how to adjust $\eta$ accordingly.

\medskip
\noindent{\bf Fact 3.} There exists $\eta_1\in(0,1)$ such that for all $\eta\in(0,\eta_1]$,  $f \in H$
and $w\in\delta\overline{\DD},$
$$\left\|\big(P(C_{\varphi})-TC_{\varphi_w}\big)(F_r)\right\|\leq\veps\|f\|.$$
\begin{proof}[Proof of Fact 3]
Since $r$ is fixed, there exists a constant $C>0$ (depending on $r$) such that $\|F_r\|\leq C\|f\|$.
 $F_r$ belongs to the finite-dimensional space $H_0=\bigoplus_{\beta\in I_r} \mathcal P_{\beta}.$ A basis of $H_0$ is given by $\bigcup_{\beta\in I_r}\mathcal B_\beta.$
The matrix of $P(C_\varphi)_{|H_0}$ with respect to this basis is block diagonal,  each block is equal to
\begin{equation}
\label{eq:block}
\begin{pmatrix}
P(\lambda^{\beta})&\displaystyle \frac{P'(\lambda^{\beta})}{1!}&\dots&\displaystyle \frac{P^{(m(\beta))}(\lambda^{\beta})}{m(\beta)!}\\
0&\ddots&\ddots&\vdots\\
\vdots&\ddots&\ddots&\displaystyle \frac{P'(\lambda^{\beta})}{1!}\\[0.2cm]
0&\dots&0&P(\lambda^{\beta})
\end{pmatrix}
\end{equation}
Indeed, we write $J(\lambda^\beta,m(\beta)+1)=\lambda^\beta I+N,$  with $N^{m(\beta)+1}=0,$ and observe that
$$P(J(\lambda^\beta,m(\beta)+1))=\sum_{k=0}^{m(\beta)}\frac{P^{(k)}(\lambda^k)}{k!}N^k.$$
The matrix of $TC_{\varphi_w}$ with respect to the same basis is also block diagonal,  each block equals
$$w^{|\beta|}\begin{pmatrix}
c_{\beta,0}&c_{\beta,1}&\dots&c_{\beta,m(\beta)}\\
0&\ddots&\ddots&\vdots\\
\vdots&\ddots&\ddots&c_{\beta,1}\\
0&\dots&0&c_{\beta,0}
\end{pmatrix}.$$
Our choice of $P$ ensures that the coefficients of the matrix of $P(C_\varphi)_{|H_0}$ can be made arbitrarily close to those of $TC_{\varphi_w|H_0}$, which proves the result.
\end{proof}

\medskip
\noindent{\bf Fact 4.} There exists $\eta_2\in(0,1)$ such that for all $\eta\in(0,\eta_2]$ and $f \in H,$
$$\left\|P(C_{\varphi})(\widetilde{F_r}-F_r)\right\|\leq\veps\|f\|.$$
\begin{proof}[Proof of Fact 4]
We write $\widetilde{F_r}-F_r=\sum_{\beta\in K_r\backslash I_r}f_\beta.$ Again, there exists $C>0$ such that, for all $\beta\in K_r\backslash I_r,$ $\|f_\beta\|\leq C\|f\|$. The matrix of $P(C_\varphi)_{|H_1}$, where $H_1=\bigoplus_{\beta\in K_r\backslash I_r}\mathcal P_\beta$ is again block diagonal, with each block given by \eqref{eq:block}. Since $|\lambda^{\beta}|\leq\rho<(r+\rho)/2$ for all $\beta\in K_r\backslash I_r,$ the coefficients of this matrix can be made arbitrarily close to $0$ by our choice of $P,$ allowing $\eta>0$ to be small enough, which proves the result.
\end{proof}

We now easily finish the proof of Proposition \ref{prop:approxim}. Let $\eta=\min(\eta_1,\eta_2)$
and $f\in H.$ For $w\in\delta\overline{\DD},$
\begin{align*}
\|P(C_\varphi)(f)-TC_{\varphi_w}(f)\|&\leq \|P(C_\varphi)(F_r)-TC_{\varphi_w}(F_r)\|+\|TC_{\varphi_w}(G_r)\|\\
&\quad\quad +\|P(C_\varphi)(\widetilde F_r-F_r)\|+\|P(C_\varphi)(\widetilde{G_r})\|\\
&\leq 4\veps\|f\|.
\end{align*}
\end{proof}

 We are now ready for the proof of Theorem \ref{thm:sufficiency}.
\begin{proof}[Proof of Theorem \ref{thm:sufficiency}]
Let $\Lambda$ be a linear functional on $H$ that is continuous for the weak-star topology and vanishes on $\{P(C_\varphi):\ P\in\CC[X]\}$. Let $T\in\{C_\varphi\}'$. We need to show that $\Lambda(T)=0.$ Define $F$ on $\Omega$ by $F(w)=\Lambda(TC_{\varphi_w})$. Arguing as in \cite[Lemma 7.6]{BY}, we may assume without loss of generality that
$\Lambda(T)=\langle Tf,g\rangle$ for some $f,g\in\fcd$. So that
\begin{equation*}
  F(w)=\langle TC_{\varphi_w}f,g\rangle=\langle C_{\varphi_w}f,T^* g\rangle.
\end{equation*} Then we have that $F$ is analytic on $\Omega.$
By Proposition \ref{prop:approxim}, $F=0$ on $\delta\mathbb D,$ hence  $F=0$ on $\Omega$ by the analyticity of $F.$ To conclude the proof, it suffices to prove that $C_{\varphi_w}$ converges weakly to $I$ in the weak operator topology
as $w\to 1$, $w\in(0,1)$.
By the uniform boundedness of $(C_{\varphi_w})_{w\in (0,1)}$, it suffices to show that $\langle C_{\varphi_w}u,v\rangle\to \langle u,v\rangle$ for all $u\in H$ and all $v$ in a total subset of $H$, e.g. for all reproducing kernels $k_z$ with $z\in\CC^d$. This is easy:
$$\langle C_{\varphi_w}u,K_z\rangle=u(\varphi_w(z))\xrightarrow{w\to 1^-} u(z)=\langle u,k_z\rangle.$$
This proves the desired result.
\end{proof}

\section{Application to composition operators on Fock space}\label{sec:fock}
\subsection{Framework}

Let $d\geq 1$, let us endow $\mathbb{C}^{d}$ with the standard inner product $\langle z,w\rangle=\sum_{j=1}^{d}z_{j} \overline{w_{j}}$ and let us denote $|z|=\sqrt{\langle z,{z}\rangle}$.
The classical Fock space on $\mathbb{C}^{d}$ is the separable Hilbert space of holomorphic functions on $\CC^d$ defined by
\begin{equation*}
\fcd=\left\{f\in H(\mathbb{C}^{d}):  \|f\|^{2}:=\frac{1}{(2\pi)^{d}}\int_{\mathbb{C}^{d}}|f(z)|^{2}e^{-{|z|^{2}}/{2}}dA(z)<\infty\right\}.
\end{equation*}
The space $\fcd$ is a reproducing kernel Hilbert space. It is well known that the reproducing kernel function at $w\in\CC^d$ is given by
\begin{equation*}
  k_{w}:z\in \mathbb{C}^{d}\mapsto k_{w}(z)=\exp\left(\frac{\langle z,w\rangle}{2}\right),
\end{equation*}
with norm $\|k_{w}\|=\exp\left(\frac{|w|^{2}}{4}\right)$.
Moreover, the set of polynomials $\{z^{\alpha}:\alpha\in \mathbb{Z}_+^{d}\}$ is an orthogonal basis of $\fcd$ with
\begin{equation}\label{eq:normpowerfock}
  \|z^{\alpha}\|^{2}=2^{|\alpha|}\alpha!
\end{equation}
for all $\alpha\in\ZZ_+^d.$
For more details about Fock spaces, we refer to Zhu's monograph \cite{Zh}.

\smallskip

A holomorphic map $\varphi:\mathbb{C}^{d}\rightarrow\mathbb{C}^{d}$
induces a composition operator $C_\varphi$ defined by $C_{\varphi}f=f\circ\varphi$ for any  $f\in \fcd$.
Carswell et al \cite{CMS} have characterized when $C_{\varphi}$ is a bounded or compact composition operator on $\fcd$, as follows.

\begin{theorem}\label{bound-com}
Let $\varphi:\mathbb{C}^{d}\rightarrow\mathbb{C}^{d}$ be holomorphic.
\begin{itemize}
  \item[(a)] $C_{\varphi}$ is bounded on $\fcd$ if and only if $\varphi(z)=Az+B$, where $A\in \mathcal M_d(\mathbb C)$, $\|A\|\leq1$, $B\in \mathbb{C}^{d}$,
and $\langle A\zeta,B\rangle=0$ for all $\zeta\in\mathbb{C}^{d}$ with $|A\zeta|=|\zeta|$.
  \item[(b)] $C_{\varphi}$ is compact on $\fcd$ if and only if $\varphi(z)=Az+B$, where $A\in \mathcal M_d(\mathbb C)$ and $\|A\|<1$, $B\in \mathbb{C}^{d}$.
\end{itemize}
\end{theorem}

For convenience, let $\mathcal{S}_{c}(\mathbb{C}^{d})$ denote the above self-maps of $\mathbb{C}^{d}$ inducing bounded composition operators on $\fcd$.
Carswell et al. \cite{CMS} have also computed the norm of $C_\varphi$:
\begin{equation}\label{norm}
  \|C_{\varphi}\|=\exp\left(\frac{|z_0|^2 -|Az_0|^2+|B|^2}{4}\right)
\end{equation}
where $z_{0}$ is any solution of $(I_{d}-A^{\ast}A)z=A^{\ast}B$.
Here and below, $A^{\ast}$ denotes the conjugate transpose of $A$.

\subsection{Statement and proof}
We characterize all compact composition operators on $\fcd$ with the minimal commutant property, and even slightly more.

\begin{theorem}\label{thm:fock}
Let $\varphi(z)=Az+B$ induce a bounded composition operator on $\fcd$ with $\rho(A)<1.$ Then $C_\varphi$ has the minimal commutant property if and only if
\begin{itemize}
  \item[$\bullet$] the canonical Jordan form of $A$ admits at most one Jordan block of size greater than $1$, and the size of this Jordan block, if it exists,  is exactly $2$;
  \item[$\bullet$] if $\lambda_1,\dots,\lambda_{\hat d}$ are the eigenvalues of $A$, listed according to their geometric multiplicities, then they are independent.
\end{itemize}
\end{theorem}
\begin{proof}
By Lemma \ref{lem:powercommutant}, we may assume without loss of generality that $\|A\|<1.$ It suffices to show that $(H, \varphi)$ is a strong Koenigs system.
Let $\xi$ be the fixed point of $\varphi$ (recall that $A-I_{d}$ is invertible) and write $\varphi(z)=A(z-\xi)+\xi.$ The map $\sigma$ is given by $\sigma(z)=z-\xi,$ and $\sigma^\alpha$ belongs to $\fcd$ for all $\alpha\in \ZZ_+^d.$ It is also clear that (SK1) hold.
Unfortunately, \eqref{eq:normpowerfock} implies that we cannot appeal to Lemma \ref{lem:suf1} to prove (SK2) and (SK3).

We begin by proving (SK3).
By the rotational invariance of the norm in $\fcd$, we may  assume that $\xi=\xi_1 e_1.$
Let $g(z)=\sum_{|\alpha|=n} a_\alpha z^{\alpha}\in \mathcal G_n.$ Then
\begin{align*}
C_\sigma(g)&=\sum_{\alpha_{1}+|\alpha'|=n} a_\alpha (z_1-\xi_{1})^{\alpha_{1}} (z'_1)^{\alpha'}\\
&=\sum_{\alpha_{1}+|\alpha'|=n} a_\alpha\sum_{j=0}^{\alpha_{1}} \binom {\alpha_{1}}j\left(-\xi_1\right)^{j}z_1^{\alpha_{1}-j} (z'_1)^{\alpha'}\\
&=\sum_{j=0}^n (-\xi_1)^{j} \sum_{\alpha_{1}+|\alpha'|=n,\alpha_{1}\geq j} \binom {\alpha_{1}}j a_\alpha z_1^{\alpha_{1}-j} (z'_1)^{\alpha'}.
\end{align*}
where $z=(z_1,z'_1)$ and $\alpha=(\alpha_1,\alpha'_1)$.
Therefore
\begin{align*}
\|C_\sigma(g)\| \leq & \sum_{j=0}^n |\xi_1|^{j} \sum_{\alpha_{1}+|\alpha'|=n,\alpha_{1}\geq j} \binom {\alpha_{1}}j |a_\alpha| \left\|z_1^{\alpha_{1}-j} (z'_1)^{\alpha'}\right\|\\
\leq & \sum_{j=0}^n |\xi_1|^{j} \binom {n+d-1}{d} ^{1/2} 2^{n}\left(\sum_{\alpha_{1}+|\alpha'|=n,\alpha_{1}\geq j} |a_\alpha|^{2} 2^{n-j}(\alpha_{1}-j)!\alpha'!\right)^{1/2}\\
\leq & M^{n}\sum_{j=0}^n \left(\sum_{\alpha_{1}+|\alpha'|=n,\alpha_{1}\geq j} |a_\alpha|^{2} 2^{n}\alpha_{1}!\alpha'!\right)^{1/2}\\
\leq & M^{n}(n+1)\|g\|\\
\leq & M^{n}\|g\|.
\end{align*}

On the other hand, let $f(z)=\sum_{|\alpha|=n}b_\alpha(z-\xi)^\alpha\in\mathcal{F}_n.$ Then
$f(z)=\sum_{|\alpha|=n} b_\alpha z^\alpha+Q(z)$ where $Q$ is a polynomial of degree (strictly) less than $n.$ Therefore,
by orthogonality,
$$\|C_{\sigma^{-1}}f\|=\left\|\sum_{|\alpha|=n}b_\alpha z^\alpha\right\|\leq\|f\|.$$

To prove (SK2), let $(x_1^*,\dots,x_d^*)$ be a basis of $(\mathbb C^d)^*$ and let $Q\in GL_d(\CC)$
such that $x^*=Qe^*$,  where $x^*=(x_1^*,\dots,x_d^*)^T$ and $e^*=(e_1^*,\dots,e_d^*)^T.$
 Let $n\geq 0$ and $\alpha\in\mathbb{Z}_+^{d}$ with $|\alpha|=n$. We claim that
\begin{equation}\label{inv-m-monomial}
\|(x^*)^{\alpha}\|\leq |\det{Q}|^{-1} \|Q\|^{n+d} \|z^{\alpha}\|
\end{equation}
Indeed, note that
\begin{equation*}
\|(x^*)^{\alpha}\|^{2}=\frac{1}{(2\pi)^{d}}\int_{\mathbb{C}^{d}}|(Qz)^{\alpha}|^{2}e^{-\frac{|z|^{2}}{2}}dA(z).
\end{equation*}
We perform the change of variable $w=Qz$. Since $Q$ is invertible,
we have $z=Q^{-1}w$ and the Jacobian of this transformation is $|\det(Q^{-1})|=|\det{Q}|^{-1}$.
Moreover, $\|Q\|^{-1}|w|\leq|z|\leq\|Q^{-1}\||w|$. Thus
\begin{align*}
\|(x^*)^{\alpha}\|^{2}
 &\leq\frac{1}{(2\pi)^{d}}|\det{Q}|^{-2}\int_{\mathbb{C}^{d}}|w^{\alpha}|^{2}e^{-\frac{\|Q\|^{-2}|w|^{2}}{2}}dA(w) \\
  &= |\det{Q}|^{-2} \frac{1}{2^{d}} \prod_{j=1}^{d}(2\|Q\|^{2})^{\alpha_{j}+1}\alpha_{j}!\\
  &= |\det{Q}|^{-2} \|Q\|^{2(n+d)} \|z^{\alpha}\|^{2},
\end{align*}
as desired. Observe also that, by Stirling's formula, there exists $M>0$ such that
$$\sup_{|\beta|=n} \|z^\beta\|\leq M^n \inf_{|\beta|=n}\|z^\beta\|$$
so that, for all $n\in\NN,$  $\alpha\in\NNZ^d$ with $|\alpha|=n,$
$$\|(x^*)^{\alpha}\|\leq CM^n \inf_{|\beta|=n}\|z^\beta\|.$$
Pick  $f$ a polynomial and rewrite it as follows
$$f=\sum_\alpha a_\alpha L^\alpha$$
(recall that $\sigma(z)=z-\xi$). We get
$$f\circ\sigma^{-1}=\sum_n \sum_{|\alpha|=n}a_\alpha (x^*)^\alpha=\sum_n \sum_{|\beta|=n}b_\beta z^\beta.$$
Arguing exactly as in the proof of Lemma \ref{lem:suf1}, we find for all $n\geq 0$ and  $\alpha\in\NNZ^d$ with $|\alpha|=n,$
$$|a_\alpha|\leq C M^n \sup_{|\beta|=n} |b_\beta|.$$
Finally, we find
\begin{align*}
\|a_\alpha L^\alpha\|&\leq |a_\alpha| \cdot \|C_\sigma((x^*)^\alpha)\|\\
&\leq CM^n \left(\sup_{|\beta|=n}|b_\beta|\right)\left(\inf_{|\beta|=n} \|z^\beta\|\right)\\
&\leq CM^n \sup_{|\beta|=n}\|b_\beta z^\beta\|.
\end{align*}
Recall that $f(z)=\sum_n \sum_{|\beta|=n} b_\beta (z-\xi)^\beta$ and consider $\tilde{\sigma}(z)=\frac z2+\xi$
which induces a bounded composition operator on $\fcd.$ Moreover
$$C_{\tilde \sigma}(f)=\sum_n \sum_{|\beta|=n} b_\beta 2^{-n} z^\beta$$
so that, for all $\beta$ with $|\beta|=n,$
$$\|b_\beta z^\beta\|\leq 2^n \|C_{\tilde \sigma}\|\cdot \|f\|.$$

Property (SK4) follows from Lemma \ref{lem:suf2} since $\|A\|<1.$
We finally prove (SK5) for $\Omega=\DD.$ The proof of \eqref{eq:sks2} is straightforward since we have an explicit formula for the norm of $C_{\varphi_w}$. Indeed, since $\varphi_w(z)=wz+(\xi-w\xi),$ by \eqref{norm} we have
$$\|C_{\varphi_w}\|=\exp\left(\frac 14\left(|\xi_0|^2-|w|^2|\xi_0|^2+|\xi-w\xi|^2\right)\right)$$
where $\xi_0$ is the solution to $(1-|w|^2)\xi_0=\bar w (1-w)\xi.$ We get that
$$\|C_{\varphi_w}\|= \exp\left(\frac 14\cdot \frac{|1-w|^2}{1-|w|^2}\cdot |\xi|^2\right).$$
For $w\in[0,1),$ this simplifies to
$$\|C_{\varphi_w}\|= \exp\left(\frac 14\cdot \frac{1-w}{1+w}\cdot |\xi|^2\right)$$
which is bounded on $[0,1)$.
It remains to prove the analyticity of $F:w\in\DD\mapsto \langle C_{\varphi_w}(f),g\rangle$ for all $f,g\in\mathcal F(\CC^d).$ Write, for each $w\in\Omega,$
$$C_{\varphi_w}(f)=\sum_{\alpha\in\NNZ^d} a_\alpha(w)z^\alpha.$$
Then the maps $w\in\DD\mapsto a_\alpha(w)$ are analytic. Indeed, for all $w\in\DD,$
$$f(z)=\sum_{\alpha\in\NNZ^d}\frac{\partial^\alpha f}{\partial z^\alpha}((1-w)\xi)\big(z-(1-w)\xi\big)^\alpha.$$
So that
$$C_{\varphi_w}f(z)=\sum_{\alpha\in\NNZ^d}\frac{\partial^\alpha f}{\partial z^\alpha}((1-w)\xi)w^{|\alpha|}z^\alpha.$$
Hence,
$$a_\alpha(w)=w^{|\alpha|} \frac{\partial^\alpha f}{\partial z^\alpha}((1-w)\xi).$$

Write  $g=\sum_{\alpha\in\NNZ^d}c_\alpha z^\alpha$ so that
$$F(w)=\sum_{\alpha\in\NNZ^d} \frac{a_\alpha(w)\overline{c_\alpha}}{2^{|\alpha|}\alpha!}.$$
Furthermore, for all $w\in\DD,$ by the Cauchy-Schwarz inequality,
\begin{align*}
\sum_{\alpha} \left| \frac{a_\alpha(w)\overline{c_\alpha}}{2^{|\alpha|}\alpha!}\right|&\leq \|C_{\varphi_w}(f)\|\cdot \| g\|    \leq  \|C_{\varphi_{w}}\|\cdot  \|f\|\cdot \|g\|\\
&=\exp\left(\frac 14\cdot \frac{|1-w|^2}{1-|w|^2}\cdot |\xi|^2\right)\|f\|\cdot\| g\|.
\end{align*}
Since the left-hand side is locally bounded, the analyticity of $F$ follows from a result of Mattner \cite{Mattner}.
\end{proof}

\begin{remark}
 The case $\|A\|=1$ remains open, even if $\varphi(z)=Az$ with $A$ diagonalizable and having at least one eigenvalue of modulus $1$ and at least one eigenvalue of modulus less than $1.$ Indeed, it is unknown whether a diagonal operator (with respect to some orthonormal basis) has a minimal commutant, when a subsequence of eigenvalues lies in the unit circle whereas another subsequence goes to $0.$ See \cite{BY} for further discussion on this problem.
 \end{remark}

 \begin{remark}
 Theorem \ref{thm:fock} should be compared with the main theorem of \cite{BT}.
 Putting them together we find that, for any $\varphi(z)=Az+B\in\mathcal S_c(\CC^d),$
 provided that the spectral radius of $A$ is less than $1,$ then $C_\varphi$ is cyclic if and only
 if it has a minimal commutant.
 \end{remark}

 \begin{remark}
 The proof of Theorem \ref{thm:fock} would be easier if we could assume that $\varphi(0)=0.$
Unfortunately we cannot do that because if $\varphi(0)\neq 0,$ there does not exist $\psi\in \mathcal S_c(\CC^d)$
 such that $\psi^{-1}\in S_c(\CC^d)$ interchanging $\varphi(0)$ and $0.$
 \end{remark}

\section{Composition operators on the Hardy spaces of the  ball and  polydisc}\label{sec:ball-polydisc}

\subsection{The Euclidean ball}

In this section, we investigate composition operators on the Hardy space $\hdbd$ on the unit Euclidean ball $\bd.$
Observe that, provided $d\geq 2,$ there are holomorphic self-maps $\varphi$ of $\bd$ which do not induce bounded composition operators on $\hdbd.$ Moreover, even on the Hardy space of the disc, the minimal commutant property is well studied only for linear fractional maps (see \cite{LLSR18}).
Therefore, it is natural to restrict ourselves to the class of linear fractional maps of $\bd,$ introduced in \cite{cmlfm}.
Since for any $z_0\in\bd,$ there exists an automorphism $U$ of $\bd$ exchanging $0$ and $z_0$, the corresponding composition operator $C_U$ is bounded on $\hdbd$, and the set of linear fractional maps (including all automorphisms) is invariant under composition, then following form Lemma \ref{sim-com}, we may and shall assume that $\varphi(0)=0.$

Hence, we consider $\varphi$ a linear fractional map of $\bd$ fixing $0.$ We may write
$$\varphi(z)=\frac{Az}{\langle z,C\rangle+1},$$
with $A\in\mathcal M_d(\CC)$ and $C\in\CC^d.$ We say that $\varphi$ is \emph{loxodromic} if
\begin{itemize}
\item Every eigenvalue of $A$ has modulus less than $1$;
\item $|P|<1$, where $P=(I_d-A^*)^{-1}C$.
\end{itemize}
It should be observed that the condition $|P|\leq 1$ is mandatory for $\varphi$ to send $\bd$ into $\bd$ (see \cite[Proof of Theorem 19]{cmlfm}). Moreover, this terminology coincides with the one-variable setting.
Recall that if $\varphi$ is a linear fractional map of $\DD,$ it is loxodromic if it is not an automorphism and if it admits one fixed point in $\DD$ and one fixed point outside $\overline\DD.$ Since $\varphi(0)=0,$ the other fixed
point is $z=-(1-A)\overline{C}^{-1}$ (now $A$ and $C$ are complex numbers), it lies outside $\overline\DD$
if and only if $|(1-\bar A)^{-1}C|<1.$

\begin{theorem}\label{thm:ball}
Let $\varphi$ be a loxodromic linear fractional self-map of $\bd$ with $\varphi(0)=0.$ Then $C_\varphi$ has the minimal commutant property on $\hdbd$ if and only if
\begin{itemize}
  \item[$\bullet$] The canonical Jordan form of $A$ admits at most one Jordan block whose size exceeds $1$, and the size of this Jordan block, if it exists,  is exactly $2$;
  \item[$\bullet$] If $\lambda_1,\dots,\lambda_{\hat d}$ are the eigenvalues of $A$, repeated by geometric multiplicity, then $\lambda_1\dots,\lambda_{\hat d}$ are independent.
\end{itemize}
\end{theorem}
\begin{proof}
It suffices to prove that $(\hdbd,\varphi)$ has the strong Koenigs property. As in the proof of Theorem \ref{thm:fock}, we may and shall assume that $\|A\|<1.$ As shown in \cite{cmlfm}, the linear fractional map
$$\sigma(z)=\frac{z}{\langle z,P\rangle+1}$$
is defined on $\bd$ and verifies $\sigma\circ\varphi=A\sigma.$ It is also easy to check that $\sigma$ is injective, and
$$\sigma^{-1}(z)= \frac{z}{\langle z,-P\rangle+1}.$$
As $\sigma$ extends continuously on $\overline{\bd},$
$\sigma$ is bounded on $\bd$ and therefore $\|\sigma^\alpha\|\leq CM^{|\alpha|}$ for all $\alpha\in\NNZ^d.$ The same inequality is satisfied for $z^\alpha$ and since $\|A\|<1,$ we may apply Lemma \ref{lem:suf1} and Lemma \ref{lem:suf2} and it remains to prove (SK5). Since $\sigma$ extends to $\overline{\bd}$ and is injective, \cite[Theorem 6]{cmlfm} tells us that $\sigma(\bd)$ is an ellipsoid containing $0.$ In particular, $\sigma(\bd)$ is strictly starlike with respect to the origin, meaning that $r\overline{\sigma(\bd)}\subset\sigma(\bd)$ for all $r\in(0,1).$ As in \cite[Lemma 6.4]{LLSR18}, this implies that if we set
$$\Omega=\{w\in\CC:\ w\overline{\sigma(\bd)}\subset\sigma(\bd)\},$$
then $\Omega$ is a domain that contains $[0,1).$ Moreover, for all $w\in\Omega\backslash\{0\}$ and all $z\in\bd,$
\begin{align*}
\varphi_w(z)&=\frac{wz}{\langle z,(1-\bar w)P\rangle+1},\\
\varphi_w^{-1}(z)&=\frac{z}{\langle z,(\bar w-1)P\rangle+w}.
\end{align*}
We show that the norm of $C_{\varphi_w}$ is controlled if $|w|$ is small enough. Let $\eta>0$ be such that $(1+\eta)\|P\|<1$ and let $\delta=1-(1+\eta)\|P\|.$ Then provided $|w|\leq\min(\eta,\delta/2),$ for any $z\in\bd,$
$$|\varphi_w(z)|\leq \frac{\delta/2}{1-(1+\eta)\|P\|}\leq \frac 12.$$
This shows that
$$\sup_{\begin{subarray}c w\in \Omega \\
|w|\leq\min(\eta,\delta/2)\end{subarray}} \|C_{\varphi_w}\|<\infty.$$
To control the norm of $\|C_{\varphi_w}\|$ for the other values of $w,$ we shall use the following result of \cite{MclSha86}: let $\psi:\bd\to\bd$ be holomorphic, univalent, and satisfying $\psi(0)=0.$ Then $C_\psi$ is bounded on $\hdbd$ and
$$\|C_\psi\|\leq C\left(\sup_{z\in\psi(\bd)} |(\psi^{-1})'(z)|\right)^{2(d-1)}.$$

To prove that $\sup_{w\in\Omega,\ |w|\geq\min(\eta,\delta/2)}\|C_{\varphi_w}\|<\infty,$ in view of the formula of $\varphi_w^{-1},$ one just need to prove that
$$|\langle z,(\bar w-1)P\rangle+w|\geq\veps$$
for some $\veps>0$  independent of $z\in\varphi_w(\bd)$ and  $w\in\Omega$ with $|w|\geq\min(\eta,\delta/2)$.
Let $w\in\Omega$ with $|w|\geq\min(\eta,\delta/2)$ and let $z=\varphi_w(z_0)\in\varphi_w(\bd).$
We have
\begin{align*}
|\langle \varphi_w(z_0),(\bar w-1)P\rangle+w|=\left|\frac{w}{\langle z_0,(1-\bar w)P\rangle+1}\right|
\geq \frac{\min(\eta,\delta/2)}{2}.
\end{align*}

Finally, let $f,g\in \hdbd$ and observe that, for all $w\in\Omega,$
$$\langle C_{\varphi_w}(f),g\rangle = \int_{\mathbb S_d}f(\sigma^{-1}w\sigma(\xi))\overline{g(\xi)}d\sigma(\xi).$$
Since $w\mapsto f(\sigma^{-1}w\sigma(\xi))\overline{g(\xi)}$ is analytic on $\Omega$ for almost all $\xi\in\mathbb S_d$ and since
$$\int_{\mathbb S_d}|f(\sigma^{-1}w\sigma(\xi))\overline{g(\xi)}|d\sigma(\xi)\leq \|C_{\varphi_w}(f)\|\cdot\|g\|$$
is locally bounded on $\Omega,$ Mattner's theorem \cite{Mattner} implies \eqref{eq:sks3}.
\end{proof}

\subsection{The polydisc}

The study of composition operators on the unit polydisc $\DD^d$ is a delicate subject and it is hard to characterize when a holomorphic self-map $\varphi$ of $\DD^d$ induces a bounded composition operator on $H^2(\DD^d)$. Therefore, we shall restrict ourselves to the affine symbols $\varphi(z)=Az+b$. It is described in \cite{BAYPOLY} how to decide if $C_\varphi$ is continuous or not on $H^2(\DD^d)$ in that case.

\begin{theorem}\label{thm:polydisc}
Let $\varphi(z)=Az+b$ be an affine self-map of $\DD^d$ and assume that $C_\varphi$ induces a bounded composition operator on $H^2(\DD^d).$ Assume moreover that $\rho(A)<1.$ Then $C_\varphi$ has the minimal commutant property if and only if
\begin{itemize}
  \item[$\bullet$] The canonical Jordan form of $A$ admits at most one Jordan block whose size exceeds $1$, and the size of this Jordan block, if it exists,  is exactly $2$;
  \item[$\bullet$] If $\lambda_1,\dots,\lambda_{\hat d}$ are the eigenvalues of $A$, repeated by geometric multiplicity, then $\lambda_1\dots,\lambda_{\hat d}$ are independent.
\end{itemize}
\end{theorem}
\begin{proof}
It suffices to prove that $(H^2(\DD^d),\varphi)$ has the strong Koenigs property. As above, we may and shall assume that $\|A\|<1.$ The Koenigs map is given by $\sigma(z)=z-\xi$ where $\xi\in\DD^d$ is the fixed point of $A.$ We deduce easily that, for all $\alpha\in\ZZ_+^d,$ $$\|\sigma^\alpha\|_{H^2(\DD^d)}\leq 2^{|\alpha|}\textrm{ and }\|z^\alpha\|=1.$$
Since $\|A\|<1$, there exists $\lambda>1$ such that $\|\lambda A\|<1$ so that $C_{\lambda A}$ is a bounded operator on $H^2(\DD^d)$. Finally, (SK5) is also clear since we can take $\Omega=\DD$ and for $w\in\Omega,$ $\varphi_w(z)=wz.$
\end{proof}

\subsection{One variable result}

Our theorem allows us to recover \cite[Theorem 6.1]{LLSR18} easily. Recall that on the unit disc, a Koenigs function always exists.

\begin{theorem}
Let $\varphi$ be a univalent, holomorphic self-map of the unit disk with $\varphi(0) =0.$ Let $\sigma$ be the Koenigs function of $\varphi$ and suppose that its range $\sigma(\DD)$ is bounded and strictly starlike with respect to the origin. Then the operator $C_\varphi$ acting on $H^2(\DD)$ has a minimal commutant.
\end{theorem}

\begin{proof}
Again, one proves that $(H^2(\DD),\varphi)$ is a strong Koenigs system.
In fact, all the hypotheses ‌imply‌ this easily. Observe that in this particular case, $A=\varphi'(0)$ satisfies $|A|<1.$
\end{proof}

\begin{remark}
Our proof shows that we can replace the assumption $\sigma(\DD)$ is bounded by {$\sigma^n\in H^2(\DD)$ for all $n\geq 0.$}
This is interesting because the latter assumption is satisfied as soon as $C_\varphi$ is compact (see \cite[p. 94]{sh}).
\end{remark}

\color{black}
\medskip
\noindent \textbf{Acknowledgment.}
M. Wang and X. Yao were supported by National Science Foundation of China
(No. 12571145).


\begin{thebibliography}{99}

\bibitem{BAYPOLY} F. Bayart,
\newblock {\em Composition operators on the polydisk induced by affine maps},
J. Funct. Anal., \textbf{260} (2011), 1969--2003.

\bibitem{BT} F. Bayart and S. Tapia-Garc\'{\i}a,
\newblock{\em Cyclicity of composition operators on the Fock space},
J. Operator Theory, \textbf{92} (2024), no. 2, 549--577.

\bibitem{BY} F. Bayart and X. Yao,
\newblock {\em K\"{o}nigs maps and commutants of composition operators on the Hardy-Hilbert space of Dirichlet series,} arXiv: 2406.19737, 2024.

\bibitem{CMS} B. Carswell, B. MacCluer and A. Schuster,
\newblock{\em Composition operators on the Fock space},
Acta Sci. Math. (Szeged), \textbf{69} (2003) 871--887.

\bibitem{cl98} B. Cload,
\newblock{\em Generating the commutation of a composition operator},
\newblock Contemp. Math., \textbf{213} (1998), 11--15.

\bibitem{cm}
C. Cowen and B. MacCluer,
\newblock {\em Composition operators on spaces of analytic functions,}
\newblock CRC Press, Boca Raton, 1995.

\bibitem{cmlfm}
C. Cowen and B. MacCluer,
\newblock {\em Linear fractional maps of the ball and their composition operators},
\newblock Acta. Sci. Math. (Szeged), \textbf{66} (2000), 351--376.

\bibitem{GLR} I. Gohberg, P. Lancaster and L. Rodman,
\newblock{\em Invariant subspaces of matrices with applications,}
SIAM, 1986.

\bibitem{GL25} M. González and F. León-Saavedra,
\newblock {\em Minimal commutant and double commutant property for analytic Toeplitz operators,}
\newblock Banach J. Math. Anal. {\bf 19} (2025), no. 3, Paper No. 37, 22 pp.

\bibitem{HJ91} R. Horn and C. Johnson,
\newblock{\em Topics in matrix analysis,}
Cambridge University Press, 1991.


\bibitem{LLSR18}
M. Lacruz, F. Le\'{o}n-Saavedra, S. Petrovic and L. Rodr\'{\i}guez-Piazza,
\newblock{\em Composition operators with a minimal commutant},
\newblock Adv. Math., \textbf{328} (2018) 890--927.

\bibitem{LQ26}
P. Lefèvre and H. Queffélec,
\newblock{\em A Primer, and Beyond, on Composition Operators on the Unit Disk},
Cambridge University Press, 2026.

\bibitem{MclSha86}
B. MacCluer and J. Shapiro,
\newblock {\em Angular derivatives and compact composition operators on the Hardy and Bergman spaces,}
\newblock Can. J. Math., \textbf{38} (1986) 878--906.

\bibitem{Mattner}
M. Mattner,
\newblock {\em Complex differentiation under the integral},
\newblock Nieuw. Arch. Wiskd., \textbf{2} (2001) 32--35.

\bibitem{Sa67} D. Sarason,
\newblock {\em Generalized interpolation in $H^\infty$,}
\newblock Trans. Amer. Math. Soc., \textbf{127} (1967) 179--203.

\bibitem{sh} J. Shapiro,
\newblock{\em Composition operators and classical function theory,}
\newblock Springer, New York, 1993.

\bibitem{SW} A. Shields and L. Wallen,
\newblock{\em The commutant of certain Hilbert spaces operators,}
\newblock Indiana Univ. Math. J., \textbf{20} (1971) 777--788.

\bibitem{Wo02} T. Worner,
\newblock{\em Commutants of certain composition operators,}
\newblock Acta. Sci. Math. (Szeged), \textbf{68} (2002), 413-432.

\bibitem{Zh} K. Zhu,
\newblock{\em Analysis on Fock spaces,}
Springer-Verlag, New York, 2012.

\end{thebibliography}
\end{document}